\documentclass[12 pt]{article}
\usepackage{amsmath, amsthm, amssymb,enumerate}
\usepackage{tikz}
\usepackage{fullpage}
\usepackage{enumitem}
\usepackage{hyperref}
\usepackage{xcolor}

\title{Partitioning the vertices of a torus  into isomorphic subgraphs}
\author{Marthe Bonamy\thanks{CNRS, LaBRI, Universit\'e de Bordeaux, France. \newline E-mail: \texttt{marthe.bonamy@u-bordeaux.fr}.}
\thanks{Supported by the ANR Project DISTANCIA (ANR-17-CE40-0015) operated by the French National Research Agency (ANR).}  \and Natasha Morrison\thanks{Department of Pure Mathematics and Mathematical Statistics, University of
Cambridge, Wilberforce Road, Cambridge CB3 0WB, UK. \newline E-mail: \texttt{morrison@dpmms.cam.ac.uk}.} 
\thanks{Supported by a Research Fellowship from Sidney Sussex College, Cambridge.}
 \and Alex Scott\thanks{Mathematical Institute, University of Oxford, Woodstock Road, Oxford OX2 6GG, United Kingdom. \newline  E-mail: \texttt{scott@maths.ox.ac.uk}.}
\thanks{Supported by a Leverhulme Trust Research Fellowship.}}

\newtheoremstyle{case}{}{}{\normalfont}{}{\itshape}{:}{ }{}

\newtheorem{thm}{Theorem}[section]
\newtheorem{lem}[thm]{Lemma}

\newtheorem{conj}[thm]{Conjecture}

\theoremstyle{definition}
\newtheorem{defn}[thm]{Definition}
\newtheorem{qn}[thm]{Question}

\newtheorem*{ack}{Acknowledgement}

\newtheorem{claim}[thm]{Claim}

\newtheoremstyle{case}{}{}{\normalfont}{}{\itshape}{\normalfont:}{ }{}

\theoremstyle{case}

\numberwithin{equation}{section}

\begin{document}

\maketitle

\begin{abstract}
Let $H$ be an induced subgraph of the torus $C_k^m$. 
 We show that when $k \ge 3$ is even and $|V(H)|$ divides some power of 
 $k$, then for sufficiently large $n$ the torus $C_k^n$ has a perfect 
 vertex-packing with induced copies of $H$.  On the other hand, disproving 
 a conjecture of Gruslys, we show that when $k$ is odd and not a prime 
 power, then there exists $H$ such that $|V(H)|$ divides some power of $k$, 
 but there is no $n$ such that $C_k^n$ has a perfect vertex-packing with 
 copies of $H$. We also disprove a conjecture of Gruslys, Leader and Tan by 
 exhibiting a subgraph $H$ of the $k$-dimensional hypercube $Q_k$, such 
 that there is no $n$ for which $Q_n$ has a perfect edge-packing with 
 copies of $H$.

\end{abstract}

\section{Introduction}

For graphs $G$ and $H$, an $H$\emph{-packing of} $G$ is a collection of vertex-disjoint subgraphs of $G$ each isomorphic to $H$. An $H$-packing of $G$ is \emph{perfect} if every vertex in $V(G)$ is covered by the $H$-packing; and {\em induced} if the copies of $H$ in the packing are also induced subgraphs. If $G$ admits a perfect $H$-packing, then every vertex of $G$ must belong to a copy of $H$, and  $|V(H)|$ must divide $|V(G)|$: we shall refer to these as the \emph{base conditions} on $G$ and $H$. 

In a general graph $G$, the base conditions are not sufficient to guarantee a perfect $H$-packing, even when $G$ is vertex-transitive
(let $H$ be a $5\times5$ grid with the central vertex removed and let $G$ be an $n\times n$ torus, where $n$ is a multiple of 24).
However, the base conditions may be enough to guarantee a perfect $H$-packing of some {\em power} $G^n$ of $G$ (here $G^n$ denotes the 
$n$th Cartesian product of $G$).

For example,  Offner~\cite{offner} asked whether the base conditions on $H$ are sufficient to guarantee a perfect $H$-packing in $Q_n$ for $n$ sufficiently large with respect to $|V(H)|$. Recall that the $n$-dimensional hypercube, denoted by $Q_n$, is $(K_2)^n$.  This question was recently resolved by Gruslys~\cite{vyt}, who proved the following attractive result.

\begin{thm}[Gruslys~\cite{vyt}]\label{th:gru1}
Let $H$ be an induced subgraph of $Q_k$, for some $k\ge1$. 
If $|H|$ is a power of 2, then there is an integer $n_0=n_0(H)$ such that $Q_n$ admits a perfect induced $H$-packing for all $n \ge n_0$. 
\end{thm}

In the infinite case,  Gruslys, Leader and Tan~\cite{imre1} proved the following beautiful result, which resolves a conjecture of Chalcraft~\cite{adam2,adam1}. 

\begin{thm}[Gruslys, Leader and Tan~\cite{imre1}]\label{imre}
Let $T$ be a non-empty finite subset of $\mathbb{Z}^k$ for some $k$, where $\mathbb{Z}^k$ is treated as a subspace of the metric space $\mathbb{R}^k$. Then, for sufficiently large $n$, the space $\mathbb{Z}^n$ can be partitioned into isometric copies of $T$.
\end{thm}

Motivated by these results, Gruslys made the following natural conjecture.

\begin{conj}[Gruslys~\cite{vyt,shovyt}]\label{vconj}
Let $G$ be a finite vertex-transitive graph and let $H$ be an induced subgraph of $G$. If $|V(H)|$ divides $|V(G)|$, then there exists a positive integer $n$ such that $G^n$ admits a perfect induced $H$-packing.
\end{conj}

Our first result proves that the conjecture holds when $G$ is a cycle of even length, or more generally a Cartesian power of an even cycle
(note that $Q_n$ can be thought of as a torus, so this extends Theorem~\ref{th:gru1}).
 
\begin{thm}\label{th:packabletori}
Let $k \ge 3$ be an even integer and let $H$ be an induced subgraph of the torus $C_k^m$ such that $|V(H)|$ divides $k^m$. There exists $n_0$ such that, for all $n \ge n_0$, $C_k^n$ admits a perfect induced $H$-packing. 
\end{thm}

However, in general Conjecture~\ref{vconj} is false. That is, the base conditions on $H$ are not always sufficient for a vertex-transitive graph to admit a perfect $H$-packing. This is demonstrated by our next theorem.

\begin{thm}\label{th:unpackabletori}
Let $k$ be an odd integer that is not a prime power. There exist $m\ge1$ and an induced subgraph $H$ of $C_{k}^m$ such that $|V(H)|$ divides $k^m$ but $C_{k}^n$ does not admit a perfect $H$-packing for any $n$. 
\end{thm}

Note that, as $C_k^n$ is also edge-transitive, Theorem~\ref{th:unpackabletori} also disproves the natural weakening of Conjecture \ref{vconj} where $G$ is both vertex-transitive and edge-transitive. 

As well as vertex-packings, it is also interesting to consider {\em edge packings} of $H$, where we insist that our copies of $H$ are edge-disjoint.
As before, there are base conditions: it is necessary that $H$ is isomorphic to a subgraph of $G$ and that $|E(H)|$ divides $|E(G)|$. 
In light of Theorem~\ref{th:gru1},  Gruslys, Leader and Tan proposed the following conjecture. 

\begin{conj}[Gruslys, Leader and Tan~\cite{vyt}]\label{vtconj}
For $k \ge 1$, let $H$ be a non-empty subgraph of $Q_k$. Then there exists a positive integer $n$ such that the edges of $Q_n$ can be covered by edge-disjoint copies of $H$ (the copies of $H$ are not required to be induced).
\end{conj}

Note that as $|E(Q_n)| = n2^{n-1}$, the divisibility condition on $H$ may only be satisfied for some particular values of $n$ (in constrast to the vertex problem, where it holds for all sufficiently large $n$ or not at all). Our final result strongly disproves Conjecture \ref{vtconj}.

\begin{thm}\label{th:counter}
For every $k \ge 5$, there exists a graph $H$ such that $Q_{k-1} \subseteq H \subseteq Q_k$ and there is no $n$ 
such that the edges of $Q_n$ can be covered by edge-disjoint copies of $H$.
\end{thm}

The rest of the paper is structured as follows. In the next subsection we define the standard notation we will use throughout. The proof of Theorem~\ref{th:packabletori} 
uses a powerful tool which was developed by Gruslys, Leader and Tan~\cite{imre1}: we introduce this result in 
Section~\ref{mainsec} and then prove Theorem~\ref{th:packabletori}. 
Section~\ref{tor2} contains the proof of Theorem~\ref{th:unpackabletori}, and Theorem~\ref{th:counter} is proved in Section~\ref{sec:edges}.
We conclude with some discussion in Section \ref{thatsall}.

\subsection{Notation}
In this subsection we collect together some notation that is used throughout the rest of the paper. 

For $m \ge 1$, the torus $C_k^m$ has vertex set $\{0,\ldots, k-1\}^m$, where two vertices $(v_1,\ldots, v_{m})$ and $(u_1,\ldots,u_{m})$ are adjacent if and only if $v_i \not= u_i$ for exactly one $i \in \{1,\ldots, m\}$ and $|v_i - u_i| \equiv 1$ (modulo $k$). By definition, when $m=0$, $C_k^0$ is a single point. 

\begin{defn}\label{basis}
For $n \in \mathbb{N}$ and $j \le n$, let $e^n_j$ denote the vector in $\{0,1\}^n$ with $i$-th co-ordinate equal to $\delta_{ij}$ (for example, $e_2^4 = (0,1,0,0)$). Let $0^n$ denote the vector of length $n$ where every co-ordinate is 0. For example $0^3 = (0,0,0)$. For integers $n,k,s$ and any $x = (x_1,\ldots,x_n) \in \{0,\ldots, k-1\}^n$ and $y = (y_1,\ldots,y_s) \in \{0,\ldots,k-1\}^s$ we write $x \times y$ to denote the vector $(v_1,\ldots,v_{n+s}) \in \{0,\ldots,k-1\}^{n+s}$ where $v_i = x_i$ for $1 \le i \le n$ and $v_i = y_{i-n}$ for $i>n$.

For $X \subseteq \{0,\ldots, k-1\}^n$, $Y \subseteq \{0,\ldots, k-1\}^s$, $\mathcal{X} \subseteq \mathcal{P}(\{0,\ldots, k-1\}^n)$ and $\mathcal{Y} \subseteq \mathcal{P}(\{0,\ldots, k-1\}^s)$ we also define: $X \times y := \{x \times y: x \in X\}$; $X \times Y:= \{x \times y: x \in X, y \in Y\}$; $\mathcal{X} \times y := \{X \times y: X \in \mathcal{X}\}$, and; $\mathcal{X} \times \mathcal{Y}:= \{X \times Y: X \in \mathcal{X}, Y \in \mathcal{Y}\}$.
\end{defn}

We will also need notation to deal with multisets. 

\begin{defn}\label{multi}
Let $\mathcal{X}$ be a collection of subsets of $C_k^m$ ($\mathcal{X}$ may be a multiset). For $t \in \mathbb{N}$, let $t\cdot \mathcal{X}$ denote the multiset containing $t$ copies of each $X \in \mathcal{X}$. So if some $X$ appears in $\mathcal{X}$ with multiplicity $s$, then it appears in $t\cdot \mathcal{X}$ with multiplicity $st$. Similarly for multisets $\mathcal{X}_1$ and $\mathcal{X}_2$, for each $X$ that appears in $\mathcal{X}_1 \cup \mathcal{X}_2$ and $i \in \{1,2\}$ let $m_i(X)$ be the multiplicity of $X$ in $\mathcal{X}_i$. Define $\mathcal{X}_1 + \mathcal{X}_2$ to be the set containing each $X \in \mathcal{X}_1 \cup \mathcal{X}_2$ with multiplicity $m_1(X) + m_2(X)$ and containing no $Y$ such that $Y \notin \mathcal{X}_1 \cup \mathcal{X}_2$.
\end{defn}

\section{Proof of Theorem~\ref{th:packabletori}}\label{mainsec}
The goal in this section is to prove that when $k\ge 3$ is even, and $H$ is any induced subgraph of $C_k^m$ such that $|V(H)|$ divides $k^m$, there exists some $n_0$ such that $C_k^n$ admits a perfect induced $H$-packing for all $n \ge n_0$. Observe that it suffices to find just one $n_0$ where $C_k^{n_0}$ admits a perfect induced $H$-packing, as the vertices in $C_k^{n_0+1}$ can be covered by $k$ disjoint copies of $C_k^{n_0}$. 

In order to prove Theorem~\ref{th:packabletori} we apply a lovely theorem of Gruslys, Leader and Tomon (Theorem~\ref{th:cover}). The essence of the theorem is as follows: if, for some $n_0$, you can find two specific covers of $G^{n_0}$, then there exists a perfect $H$-packing of $G^n$ for some $n \ge n_0$. This is helpful because the two covers are often easier to find than the perfect $H$-packing itself. To state the theorem itself we require the following definition.

Given a collection $\mathcal{H}$ of copies of $H$ (where the same copy may be included multiple times), the \emph{weight} $\mathcal{H}(v)$ of a vertex $v$ with respect to $\mathcal{H}$, is the number of members of $\mathcal{H}$ in which it is contained (counted with multiplicity). 

Given a graph $G$ and a subgraph $H \subseteq G$, a collection $\mathcal{H}$ of copies of $H$ in $G$ (where the same copy may be included multiple times) is called an \emph{a-cover} (or an ($a$ mod $b$)-cover) if every vertex of $G$ has weight $a$ (or $a$ mod $b$) with respect to $\mathcal{H}$. Note that for an ($a$ mod $b$)-cover it is not necessary for each vertex of $G$ to be contained in the same number of copies of $H$, only the same number modulo $b$.

If $G$ admits a perfect $H$-packing, then clearly for any values of $a$ and $b$, there exists an $a$-cover or ($a$ mod $b$)-cover of $G$ with copies of $H$. The result of Gruslys, Leader and Tomon~\cite{imreagain} (inspired by a special case proved by Gruslys, Leader and Tan~\cite{imre1}) implies the converse statement.

\begin{thm}
\label{th:cover}
Let $G$ be a graph and let $H$ be a subgraph of $G$. If for some $r \ge 1$ there exists an $r$-cover and a (1 mod $r$)-cover of $G$ both consisting of induced copies of $H$, then there exists $n >0$ such that $G^n$ admits a perfect induced $H$-packing.
\end{thm}

In light of Theorem~\ref{th:cover}, to prove Theorem~\ref{th:packabletori} it suffices to construct for some $n_0$, an $r$ cover and a (1 mod $r$)-cover of $C_k^{n_0}$ into induced copies of $H$. Theorem~\ref{th:cover} then guarantees a perfect induced $H$-packing of $C_k^{n_1}$, for some $n_1 \ge n_0$ and 
hence a perfect induced $H$-packing of $C_k^n$ for any $n \ge n_1$.

Following the approach of Gruslys~\cite{vyt}, we will use $r:= |V(H)|$. As shown in \cite{vyt}, it is easy to construct an $r$-cover from `translates' of $H$. 

\begin{defn}[Translate]\label{trans}
For $w = (w_1,\ldots, w_m) \in \{0,\ldots,k-1\}^m$ define $T_w: \{0,\ldots,k-1\}^m \rightarrow \{0,\ldots,k-1\}^m$ as follows:
$$T_w((v_1,\ldots,v_m)) := (v_1 + w_1, \ldots, v_m + w_m),$$
where co-ordinate addition is taken modulo $k$. For $H$ a subgraph of $V(C_k^m)$, define $T_w(H):= \{x + w: x \in H\}$ and say that $T_w(H)$ is a translate of $H$. We also use $T_w(H)$ to denote the induced subgraph of $H$ with vertex set $T_w(H)$. For $\mathcal{H}$ a collection of copies of $H$, define $T_w(\mathcal{H}) := \{T_w(H): H \in \mathcal{H}\}$. Similarly, say that $T_w(\mathcal{H})$ is a \emph{translate} of $\mathcal{H}$.
\end{defn}

Observe that for $H$ an induced subgraph of $C_k^m$, for any $w \in \{0,\ldots,k-1\}^m$, $T_w(H)$ is an induced copy of $H$ in $C_k^m$. Trivially, the set of all translates gives an $r$-cover.

\begin{claim}[Gruslys~\cite{vyt}]\label{lem:cover}
Let $H$ be an induced subgraph of $C_k^m$. Then there exists a collection $\mathcal{H}$ of translates of $H$ that is a $|V(H)|$-cover of $C_k^m$ into induced copies of $H$.
\end{claim}

Thus to complete the proof of Theorem~\ref{th:packabletori} it suffices to find a collection of copies of $H$ that is a (1 mod $r$)-cover of a sufficiently high dimensional torus. This is where the main difficulty of the proof lies and will be addressed in the next subsection.

\subsection{Finding a (1 mod $r$)- cover}

To find a (1 mod $r$)-cover, we actually prove a result for more general $H$. Say that a subgraph $\emptyset \not= H \subseteq C_k^m$ \emph{does not wrap} if either $m=0$ or if there exists a translate $w \in \{0,\ldots, k-1\}^m$ such that $T_w(V(H)) \subseteq \{0, \ldots,k-2\}^m$. For $H \subseteq C_k^m$ that does not wrap, fix $T^*(H)$ to be a translate of $H$ such that $H \subseteq \{0,\ldots,k-2\}^m$ (note that there may be many choices for $T^*(H)$. We will see that we can find a collection of copies of $H$ that is a (1 mod $r$)-cover of $C_k^n$ for any $H \subseteq C_k^m$ that does not wrap in $C_k^m$. (So we don't need the condition that $|V(H)|$ divides $|V(C_k^n)|$ here.) 

Our proof will proceed by induction. However, in order to make the induction go through, we require some control over the types of copy that are used in the inductive hypothesis. Let us define the required concepts to introduce these special copies.  

 One such type is the translates of $H$, defined in Definition~\ref{trans}. The other type we care about is a \emph{bend}.  

\begin{defn}
[Bend]
For $i \in [m]$, $j \in \{1,\ldots,k-1\}$, $n > m$ and $s \in \{1,\ldots,n-m\}$ define $S^{n,s}_{i,j}: \{0,\ldots,k-1\}^m \rightarrow \{0,\ldots,k-1\}^n$ as follows:

$$S^{n,s}_{i,j}((v_1,\ldots,v_m)) = \begin{cases}
(v_1,\ldots,v_m)\times 0^{n-m} & \text{ if } v_i < j, \\
(v_1,\ldots,v_{i-1},j-1,v_{i+1},\ldots,v_m) \times t \cdot e_s^{n-m} & \text{ if }v_i = j - 1 + t. 
\end{cases}
$$
(Recall the definition of $e_s^{n-m}$ from Definition~\ref{basis}.)

For $H$ a subgraph of $C_k^m$, define the $(i,j,s)$-\emph{bend of $H$ in $C_{k}^n$} to be the set 
$$S^{n,s}_{i,j}(H):= \{S^{n,s}_{i,j}(x): x \in H\}.$$
For $\mathcal{H}$ a collection of copies of $H$, define $S^{n,s}_{i,j}(\mathcal{H}):= \{S^{n,s}_{i,j}(H): H \in \mathcal{H}\}$ and say that $S^{n,s}_{i,j}(\mathcal{H})$ is a \emph{bend} of $\mathcal{H}$. Say that a bend is a \emph{good bend} if it is applied to a graph $H \subseteq \{0,\ldots,k-2\}^m$.
\end{defn}

The $(i,j,s)$-bend of $H$ in $C_{k}^n$ can be thought of as `bending' $H$ at the $j$th level of the $i$th co-ordinate direction into the $(m+s)$th co-ordinate direction. Observe that for $H \subseteq \{0,\ldots, k-2\}^m$, the $(i,j,s)$-bend of $H$ in $C_{k}^n$ is an induced copy of $H$ in $C_k^n$. This is because all coordinates except the $i$th and $(m+s)$th are unchanged under $S^{n,s}_{i,j}$ and any vertices differing by 1 in the $i$th coordinate will differ by 1 in either the $i$th or $(m+s)$th coordinate under $S^{n,s}_{i,j}$. See Figure~\ref{shiftpic} for two examples of (good) bends. 

\begin{figure}[htbp]
\centering
\includegraphics[width=1\textwidth]{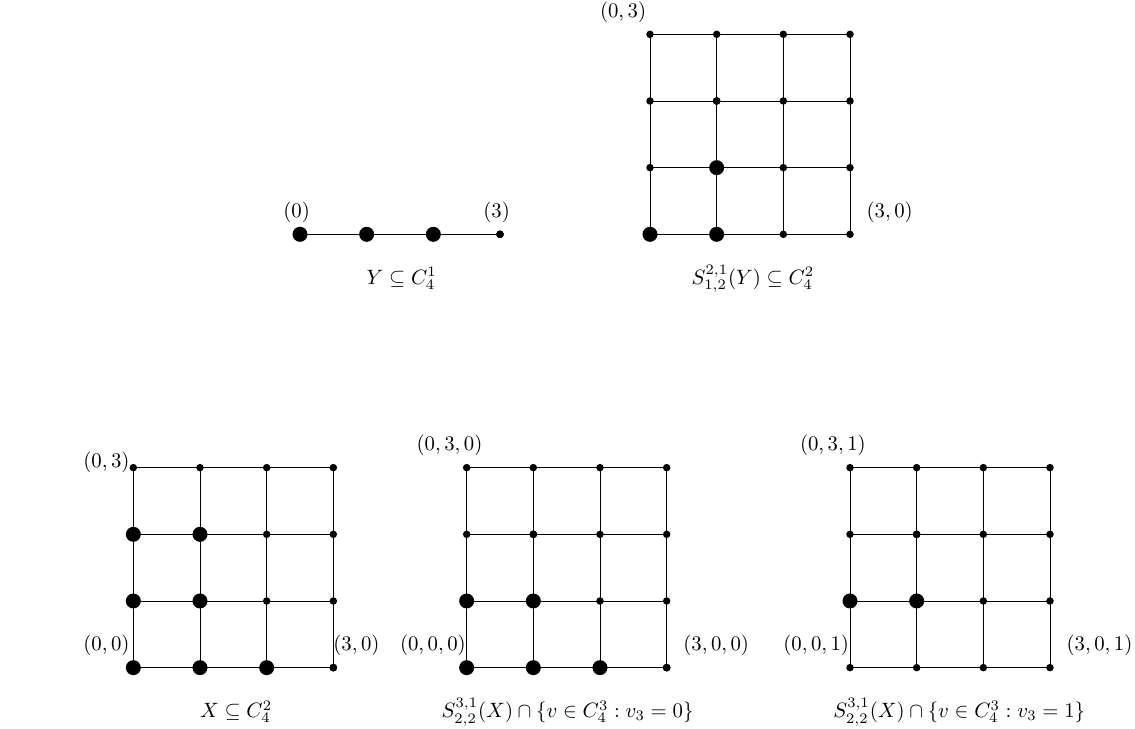}
\caption{An example of the $(1,2,1)$-bend in $C_4^2$ of a set $Y \subseteq C_4^1$ and an example of the $(2,2,1)$-bend in $C_4^3$ of a set $X \subseteq C_4^2$. By definition of a bend, $S_{2,2}^{3,1}(X) \cap \{v \in C_4^3: v_3 = 2\}$ and $S_{2,2}^{3,1}(X) \cap \{v \in C_4^3: v_3 = 3\}$ are empty. }
\label{shiftpic}
\end{figure}


Note that when $H$ does not wrap, then applying any sequence of translates and good bends results in an induced copy of $H$ that does not wrap. For $H \subseteq C_k^m$, say that $H' \subseteq C_k^{n}$ is a \emph{restricted} copy of $H$ if it can be obtained from $H$ by applying a sequence of translates and good bends. Notice that the restricted copy of $H$ is an induced copy of $H$.
The remaining tool in the proof of Theorem~\ref{th:packabletori} is the following lemma.

\begin{lem}\label{lem:evenpack}
Fix $r >0$. Let $H$ be a subgraph of $C_k^m$ that does not wrap. Then, for some $m' \ge m$, there exists a collection $\mathcal{H}$ of restricted copies of $H$ that is a (1 mod $r$)-cover of $C_k^{m'}$.
\end{lem}

We will first complete the proof of Theorem~\ref{th:packabletori} using Lemma~\ref{lem:evenpack} before proving the lemma itself.

\begin{proof}[Proof of Theorem~\ref{th:packabletori}]
For $k\ge 3$ even, let $H$ be a subgraph of $C_{k}^m$ such that $|V(H)|$ divides $k^m$. As $k$ is even, there exists some $n = n(k)$ such that $C_k$ is an induced subgraph of $Q_{n}$. It follows that $C_k^m$ is an induced subgraph of $Q_{nm}$ (e.g.~if $V = \{v_1,\ldots,v_k\}$ induces a $C_k$ in $Q_n$, then $\{w_1 \times w_2 \times \cdots \times w_m: w_i \in V\}$ induces a $C_k^m$ in $Q_{nm}$). As $k \ge 3$, $Q_{nm}$ is an induced subgraph of $C_k^{nm}$ that does not wrap, so there exists an induced subgraph $\widetilde{H}$ isomorphic to $H$ in $C_k^{nm}$ that does not wrap. Therefore, we may apply Lemma~\ref{lem:evenpack} along with Claim~\ref{lem:cover} to $\widetilde{H}$ to find an $r$-cover and (1 mod $r$)-cover of $C_k^{n_0}$ (for some $n_0\ge nm$), where both covers consist of restricted copies of $H$. Theorem~\ref{th:cover} then gives the required perfect induced $H$-packing of $C_k^{n_1}$, for some $n_1 \ge n_0$.
\end{proof}

Observe that our proof requires $k \ge 3$ to find some $\widetilde{H}$ isomorphic to $H$ in $C_k^{m'}$ such that $\widetilde{H}$ does not wrap. 

For integers $k >2$, $m\ge 1$ and $j \in \{0,\ldots,k-1\}$, define $C_k^m(j):= \{v \in C_k^m: v_1 = j\}$. Observe that $C_k^m(j)$ is isomorphic to $C_k^{m-1}$. Throughout the proof of Lemma~\ref{lem:evenpack}, it will be convenient to consider covers of $C_k^m$ into restricted copies of $H$ such that, for each $j \in \{0,\ldots,k-1\}$ there exists  $a_j$ such that every vertex of $C_k^{m}(j)$ is contained in ($a_j$ mod $r$) copies of $H$. This notion is formalised in the next definition. 

\begin{defn}[Layered cover]\label{def:LayeredPartition}
For $H \subseteq C_k^m$, define a \emph{$(a_1,\ldots,a_k)$-layered cover} of $C_k^m$ to be a collection of copies of $H$ such that for each $j \in \{0,\ldots,k-1\}$, every vertex of $C_k^m(j)$ is contained in ($a_j$ mod $|V(H)|$) copies of $H$. 
\end{defn}

The following lemma shows that to prove Lemma~\ref{lem:evenpack}, it suffices to construct a particular $(1,0,\ldots,0,-1)$-layered cover of $C_k^m$ for some $m$.

\begin{lem}\label{lem:layered}
Let $H \subseteq C_k^m$ and let $\mathcal{Z}$ be a collection of restricted copies of $H$ that is a $(1,0,\ldots,0,-1)$-layered cover of $C_k^m$. Then, for some $n \ge m$, there exists a collection $\mathcal{H}$ of restricted copies of $H$ that is a (1 mod $r$)-cover of $C_k^{n}$.
\end{lem}

\begin{proof}
Define $t:= k^m/|V(H)|$. Let $\mathcal{X}:= \{T_v(V(H)): v \in \{0,\ldots,k-1\}^m\}$ and let $\mathcal{W}^* :=  t \cdot \mathcal{X}$. By construction, every vertex $v \in C_k^m$ satisfies $\mathcal{X}(v) = |V(H)|$ and so also satisfies $\mathcal{W}^*(v) = k^m$. 

We will now construct a family of layered covers that will will add to $\mathcal{W}^*$ to `spread out' its weight into a higher dimensional torus. 

\begin{claim}\label{cl:layer}
For any $j \ge m$, there exists a collection $\mathcal{Z}^j$ of restricted copies of $H$ that is a $(1,\ldots,1,1-k)$-layered cover of $C_k^{j}$. 
\end{claim}
\begin{proof}
The proof will proceed by induction on $j$. First, consider the base case $j=m$. Let $e:= e_1^m$ and define:
$$\mathcal{Z}^m:= \sum_{a \in \{0,\ldots,k-2\}}T_{a \cdot e }((k-1-a)\cdot \mathcal{Z}).$$
For each $a \in \{0,\ldots,k-2\}$, every vertex of $C_k^m(a)$ is contained in ($k-1-a$ mod $r$) members of $T_{a \cdot e}((k-1-a)\cdot \mathcal{Z})$ and is contained in ($-(k-1-(a+1))$ mod $r$) members of $T_{(a+1) \cdot e}((k-1-(a+1))\cdot \mathcal{Z})$. Each vertex of $C_k^m(k-1)$ is contained in ($-(k-1)$ mod $r$) members of $T_0((k-1)\cdot \mathcal{Z}) = (k-1)\cdot \mathcal{Z}$. Thus $\mathcal{Z}^m$ is a $(1,\ldots,1, 1-k)$-layered cover of $C_k^m$, as required.

Now for the inductive step. For $j \ge m$ let $\mathcal{Z}^j \subseteq \{0,\ldots,k-1\}^{j}$ be a collection of restricted copies of $H$ that is a $(1,\ldots,1,1-k)$-layered cover of $C_k^j$. Define:
 $$\mathcal{Z}^{j+1}:= \sum_{a=0}^{k-1} \mathcal{Z}^j \times a.$$
We can think of $\mathcal{Z}^{j+1}$ as being the sum of $k$ `translates' of $\mathcal{Z}^{j}$, one for for each possible value of the $(j+1)$st coordinate. So $\mathcal{Z}^{j+1}$ is a $(1,\ldots,1,1-k)$-layered cover of $C_k^{j+1}$, as required.
\end{proof}

We now use the families $\mathcal{Z}^j$ given by Claim~\ref{cl:layer} to prove the following.
\begin{claim}\label{claim}
For $0 \le i < m$, let $\mathcal{W}$ be a collection of restricted copies of $H$ in $C_k^{m+i}$ that satisfies $\mathcal{W}(v) \equiv k^{m-i} \text{ mod } r$ for all $v \in C_k^{m+i}$. Then there exists a collection $\mathcal{W}'$ of restricted copies of $H$ in $C_k^{m+i + 1}$ such that $\mathcal{W}'(v) \equiv k^{m-i-1}$ mod $r$ for all $v \in C_k^{m+i + 1}$. 
\end{claim}
Repeatedly applying Claim~\ref{claim} to $\mathcal{W}^*$ gives a family satisfying the hypothesis of the lemma. Thus it remains to prove Claim~\ref{claim}.

\begin{proof}[Proof of Claim~\ref{claim}]
$C_k^{m+i + 1}(k-1)$ is a copy of $C_k^{m+i}$ and so we can use $\mathcal{W}$ to construct a collection $\mathcal{W}_0$ of restricted copies of $H$ in $C_k^{m+i +1}(k-1)$ that satisfies:
$$\mathcal{W}_0(v)= 
\begin{cases}
k^{m-i} \text{ mod } r& v_1 = k-1\\
0 \text{ mod } r& \text{ otherwise,}
\end{cases} $$
for all $v \in C_k^{m+i +1}$. To do this, just take a copy of $\mathcal{W}$ (all of whose copies of $H$ lie in $C_k^{m+i}$) and put it in $C_k^{m + i + 1}(k-1)$.

Let $\mathcal{Z}^{m+ i + 1}$ be the collection of restricted copies of $H$ that is a $(1,\ldots,1,1-k)$-layered cover of $C_k^{m+i +1}$ (as constructed in Claim~\ref{cl:layer}).
Then $\mathcal{W}' := \mathcal{W}_0 + k^{m-i - 1}\cdot \mathcal{Z}^{m+ i + 1}$ is a collection of restricted copies of $H$ that satisfies $\mathcal{W}'(v) \equiv k^{m-i-1} \text{ mod } r$ for all $v \in C_k^{m+i +1}$. 
\end{proof}
This completes the proof of the lemma.
\end{proof}

We will now use Lemma~\ref{lem:layered} to prove Lemma~\ref{lem:evenpack}. It may be helpful to recall Definition~\ref{multi}.

\begin{proof}[Proof of Lemma~\ref{lem:evenpack}]

We will prove the statement by induction on $m$. Consider the base case $m=0$. In this case $H$ consists of a single vertex and for any $n \ge 0$ there exists a collection of restricted copies of $H$ that is a (1 mod $r$)-cover of $C_k^n$. 

So now fix $m > 0$, and suppose the lemma holds for all integers less than $m$. Let $H$ be an induced subgraph of $C_k^m$ that does not wrap. We may assume, without loss of generality, that $H \subseteq \{0,\ldots,k-2\}^m$. For $j \in \{0,\ldots,k-1\}$ define $H_j := H \cap C_k^m(j)$.  We may assume that $\{j : H_j \not=\emptyset\}| \ge 2$ (else we can consider $H$ as a subgraph of $C_k^{m-1}$ that does not wrap and are done by induction). Let $\tau$ be maximal such that $C_k^{m}(\tau) \cap H$ is non-empty.

By definition, 
$$H_{\tau} \subseteq \{\tau\} \times \{0,\ldots,k-1\}^{m-1} \cong C_k^{m-1}.$$

Define:
$$H_{\tau}':= \{(v_2,\ldots,v_m): (v_1,\ldots,v_m) \in H_{\tau}\} \subseteq C_k^{m-1}.$$
By definition of `does not wrap', since $H$ does not wrap in $C_k^m$, $H_{\tau}'$ does not wrap in $C_k^{m-1}$ ($H_{\tau}'$ is defined in such a way that it has one dimension fewer than $H$). Therefore, by the inductive hypothesis, there exists some $m_0 \ge m -1$ and a collection $\mathcal{H}_{\tau}'$ of restricted copies of $H_{\tau}'$ in $C_k^{m_0}$ that is a (1 mod $r$)- cover of $C_k^{m_0}$. 

Define:
$$\mathcal{H}_{\tau}:= \{ \{\tau\} \times X: X \in \mathcal{H}_{\tau}'\}.$$ 
Thus $\mathcal{H}_{\tau}$ is a collection of restricted copies of $H_{\tau}$ contained within $C_k^{m_0+1}(\tau)$ that is a (1 mod $r$)- cover of $C_k^{m_0+1}(\tau) = \{\tau\} \times C_k^{m_0} \subseteq C_k^{m_0 + 1}.$ We wish to extend this to a collection of restricted copies of $H$ that is a (1 mod $r$)-cover of $C_k^{m_0+1}(\tau)$.

By definition, every $Y \in \mathcal{H}_{\tau}$ can be obtained from ${H}_{\tau}$ by a sequence of translates and good bends. More precisely, we have $Y = \phi_Y(H_{\tau})$, where $\phi_Y = \phi_1 \ldots \phi_t$ such that each $\phi_i$ is either a translate or a good bend (i.e.~every coordinate of $\phi_{i+1}\cdots \phi_{t}(H_{\tau})$ is in $\{0,\ldots,k-2\}$). In addition, each $\phi_i$ has the form $id \times \phi_i'$, for $\phi_i'$ some translate or bend of $C_k^{m_0}$ (i.e.~$\phi_i$ does not affect the first dimension). By construction, for each $Y$, $\phi_Y: \{0,\ldots,k-1\}^{m_0 + 1} \rightarrow \{0,\ldots,k-1\}^{m_0 +1}$ acts as the identity on the first co-ordinate. 

We now apply some sequence of good bends and translates to $H$ instead of just $H_{\tau}$. As $H$ does not wrap in $C_k^{m}$, translates and good bends of $H$ induce a subgraph isomorphic to $H$. Thus we have that, for any $Y \in \mathcal{H}_{\tau}$, the map $\phi_Y$ satisfies the following properties: 
\begin{itemize}
\item[(P1)] $\phi_Y(H)$ is isomorphic to $H$,
\item[(P2)] $\phi_Y(H_j) \subseteq C_k^{m_0 + 1}(j)$, for all $j \in \{0,\ldots,k-1\}$.
\end{itemize}

Therefore, 
$$\mathcal{X}:= \left\{\phi_Y\left(H\right): Y \in \mathcal{H}_{\tau}\right\}$$
is a collection of restricted copies of $H$ in $C_k^{m_0+1}$ such that every vertex of $C_k^{m_0 +1}(\tau)$ is contained in (1 mod $r$) members of $\mathcal{X}$.

Let $n_0:= m_0 + r$. We will now define a new collection created from bends of $\mathcal{X} \times 0^{r-1}$. For $d \in [r-1]$, define  $\mathcal{X}_d:= S^{n_0,d}_{1,\tau}(\mathcal{X})$. Define: $$\mathcal{Y}:= \mathcal{X}\times 0^{r-1} +  \sum_{d \in [r-1]}\mathcal{X}_d.$$
So $\mathcal{Y}$ is created from $r$ copies of $\mathcal{X}$ in such a way that these copies agree in the first $\tau - 1$ levels, and the final level of each copy is bent so that it sticks out in a different direction. It is easy to check that every vertex $v \in C_k^{n_0}$ satisfies: 
$$\mathcal{Y}(v) =  \begin{cases}
  1 \text{ mod } r & \text{if } v_1 = \tau \text{ and }(v_{m_0 +2},\ldots,v_{n_0}) = 0^{r-1},\\
  1 \text{ mod } r & \text{if } v_1 = \tau - 1  \text{ and } (v_{m_0 + 2},\ldots, v_{n_0}) = e^{r-1}_d \text{ for some } d\in [r-1],\\    
  0 \text{ mod } r &\text{ otherwise.} \end{cases}
$$

Define
$$W := \{0^{m_0 +1}\times v: v \in \{0,\ldots,k-1\}^{r-1}\}.$$
Also let $e:= e_1^{n_0}$ and define
$$\mathcal{Z} := \bigcup_{v \in W} T_{v - \tau\cdot e}(\mathcal{Y}).$$
By construction, every vertex $v \in C_k^{n_0}$ satisfies:
$$\mathcal{Z}(v) =  \begin{cases}
  1 \text{ mod } r & \text{if }v_{1} = 0,\\
  -1 \text{ mod } r & \text{if } v_{1} = k -1, \\    
  0 &\text{ otherwise.} \end{cases}
$$
Thus $\mathcal{Z}$ is a collection of restricted copies of $H$ that is a $(1,0,\ldots,0,-1)$-layered cover of $C_k^{n_0}$. Therefore, by Lemma~\ref{lem:layered}, there exists some $N \ge n_0$ and a collection $\mathcal{H}$ of restricted copies of $H$ that is a (1 mod $r$)-cover of $C_k^{N}$.
\end{proof}

We remark that it was important for the inductive step that we used a collection of restricted copies of $H_{\tau}$ to cover $C_k^{m_0}$. If we were allowed any copies of $H_{\tau}$, then we may not have been able to extend our cover to a well-behaved cover of $C_k^{m_0 + 1}(\tau)$. Suppose we had used a copy of $H_{\tau}$ that came from an `upside down' copy of $H$. Then when we extended this to a copy of $H$, the copy of $H$ may intersect $C_k^{m_0 +1}(\tau + 1), \ldots, C_k^{m_0 +1}(k-1)$. In this case we would not then be able to apply a bend to the resulting cover, as this particular copy of $H$ would have a co-ordinate equal to $k-1$. 

\section{Proof of Theorem~\ref{th:unpackabletori}}\label{tor2}
In this section we construct, for any $k$ a product of two odd coprime integers, a graph $H$ satisfying the base conditions for which there exists no $H$-packing of $C_{k}^n$ for any $n$. 

\begin{proof}[Proof of Theorem~\ref{th:unpackabletori}]
Let $k=a\cdot b$, where $a \not= 1$ and $b \not= 1$ are odd coprime integers. We will construct $H$ to be an induced subgraph of $C_{ab}^2$ such that $|V(H)| = a^m$ for some $m$ (and hence $H$ is an induced subgraph of $C_{ab}^m$ with $|V(H)|$ dividing $|V(C_{ab}^m)|$). 

Without loss of generality, assume that $a < b$. Choose an integer $t$  such that 
$2b-1 \le a^t < b^2$. (If $a \le b/2$ then such an integer exists, as $b^2/(2b-1) > a$. If $a > b/2$ and no such integer exists then $b^2/4 < a^2 < 2b-1.$ When $b \ge 8$, this is a contradiction. When $b \le 7$, it is not difficult to check that $a^2 \ge 2b-1$ for any valid choice of $a$, a contradiction.) 

For a vertex $v$ in $C_{ab}^2$, define the \emph{box} $B_v$ to be the set of vertices $\{v + (i,j): i,j \in \{0,\ldots,a-1\}\}$, where co-ordinate addition is taken modulo $ab$. Thus $B_v$ contains $a^2$ vertices, including $v$ itself. Let $U:= \{v \in V(C_{ab}^2): v_1 \equiv v_2 \equiv 0 \text{ mod } a\}$. The collection of boxes $\mathcal{B}:=\{B_v: v \in U\}$ partition the vertices of $C_{ab}^2$ into $b^2$ boxes. 

Let $U':= \{v=(v_1,v_2) \in U: v_1 \not=0, v_2 \not= 0\}$. Let $W$ be the first $a^{t} - (2b - 1) < (b-1)^2$ vertices of $U'$ under the lexicographic ordering (this is possible by choice of $t$). Define $H$ to be the subgraph of $C_{ab}^2$ induced by the vertices 
\begin{equation}\label{Hdef}
\left(\bigcup_{\substack{v \in U\\ v_1 = 0}}B_v\right) \cup \left(\bigcup_{\substack{v \in U\\ v_2 = 0}}B_v\right) \cup \left(\bigcup_{w \in W}B_w\right).
\end{equation}
As $H$ is the union of $a^t$ disjoint boxes, $H$ contains $a^{t+2}$ vertices. See Figure~\ref{counter} for an example of such an $H$.

\begin{figure}
\centering
\includegraphics[width=0.5\textwidth]{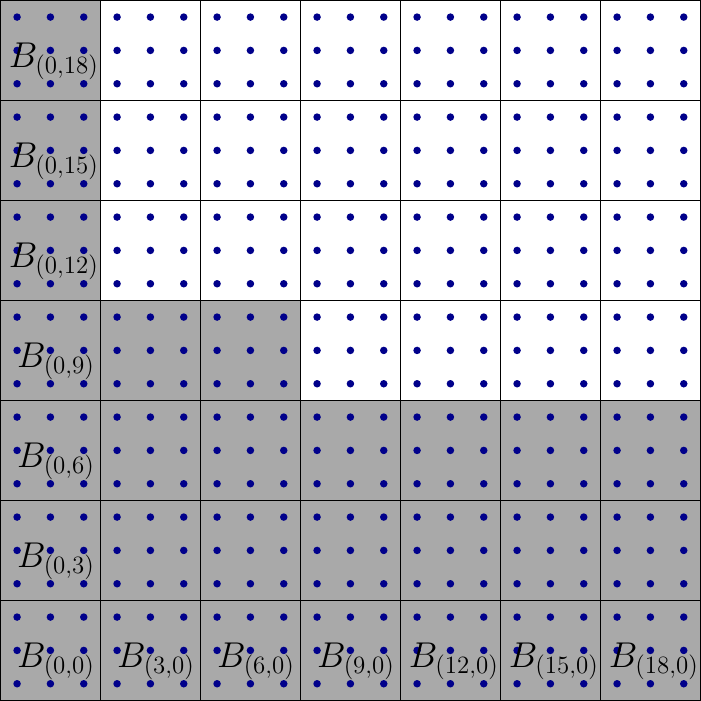}
\caption{An example of the graph $H \subseteq C_{21}^2$ constructed when $a=3$ and $b=7$. Precisely $a^3 = 27$ boxes of vertices are contained in $H$.}
\label{counter}
\end{figure}

We can partition the vertices of $C_{ab}^n$ into $a^n$ disjoint equivalence classes, by considering the value of each co-ordinate mod $a$. More formally, $(x_1,\ldots,x_n) \sim (y_1,\ldots,y_n)$ if and only if $x_i \equiv y_i$ mod $a$ for all $i \in [1,n]$. Each equivalence class has cardinality $b^n$.

The main point in the proof is to show that, for any $n$, any copy of $H$ in $C_{ab}^n$ intersects each equivalence class on 0 modulo $a$ points. As $a$ and $b$ are coprime, this implies that a particular equivalence class cannot be covered by vertex-disjoint copies of $H$. This then concludes the proof of the theorem.

First observe that every box in $C_{ab}^2$ contains exactly one point of each of the $a^2$ equivalence classes. As $H$ consists of $a^t$ disjoint boxes, it intersects each equivalence class in $C_{ab}^2$ on 0 modulo $a$ points.

We will now show that in any copy of $H$, two vertices may differ only in two coordinates. Let $\phi$ be any isomorphism from $H$ to a subgraph  $H'$ of $C_{ab}^n$. As $C_{ab}^n$ is vertex-transitive, we may assume that $\phi(0,0) = (0,0)\times 0^{n-2}$, $\phi(0,1) = (0,1) \times 0^{n-2}$ and $\phi(1,0) = (1,0) \times 0^{n-2}$. We will show that $\phi(u) = u \times 0^{n-2}$ for every $u \in V(H)$. This along with the conclusion of the previous paragraph implies that any copy of $H$ in $C_{ab}^n$ intersects each equivalence class on 0 modulo $a$ points, as required. 

$H$ contains two cycles $C_1$ and $C_2$ of length $ab$ that intersect $(0,0)$. 
As $ab$ is odd, the only cycles of length $ab$ in $C_{ab}^n$ are those given by $\{v + \ell\cdot e_i^n: \ell \in \{0,\ldots,ab-1\}\}$ for some $i \in \{1,\ldots,n\}$. Therefore for each vertex $w \in C_1 \cup C_2$, $\phi(w) = w \times 0^{n-2}$. (Once one edge of the cycle is fixed by $\phi$ there is only one choice for the others.) 

Let $v_1v_2v_3v_4$ be a copy of $C_4$ in $H$. If $\phi(v_i) = v_i \times 0^{n-2}$ for $i \in \{1,2,3\}$, then $\phi(v_4) = v_4 \times 0^{n-2}$. Thus to show that $\phi(u) =u \times 0^{n-2} $ for all $u \in V(H)$, it suffices to observe that there exists an ordering $v_1,\ldots, v_{a^{t+2}}$ of $V(H)$ such that:
\begin{itemize}
\item $\{v_1, \ldots, v_{2ab - 1}\} = V(C_1 \cup C_2)$, and
\item for each $i \ge 2ab$, there exist $j_1, j_2, j_3 < i$ such that $\{v_i, v_{j_1}, v_{j_2}, v_{j_3}\}$ is a copy of $C_4$ in $H$.
\end{itemize}
Hence as vectors, $\phi(u)= u \times 0^{n-2}$ for all $u \in V(H)$. This completes the proof of Theorem~\ref{th:unpackabletori}.

\end{proof}

We briefly draw attention to the reason this proof works for $ab$ a product of two odd coprime integers and not for $k$ even. There are many even cycles of length $k$ in $C_{k}^n$, but the only cycles of length $ab$ in $C_{ab}^n$ are those obtained by changing a single co-ordinate (i.e.~of the form $\{v + \ell\cdot e_i^n: \ell \in [0,ab-1]\}$, for some $v \in C_k^n$ and some $i \in [n]$). So in the $ab$ case, we have a lot of control over which subgraphs in $C_k^n$ are isomorphic to our choice of $H$. 

\section{Proof of Theorem~\ref{th:counter}}\label{sec:edges}

Before proceeding with the proof of Theorem~\ref{th:counter}, we state a few preliminary definitions. For an edge $uv$ in $E(Q_d)$, the \emph{direction} of $uv$, denoted $D(uv)$ is the unique co-ordinate $i$ such that $u_i \not= v_i$. For a subgraph $H \subseteq Q_d$, let $E_i(H)$ be the set of edges of $H$ with direction $i$ and let $D(H):= \{i : E_i(H) \not= \emptyset \}$. Say that a subgraph $H \subseteq Q_d$ is \emph{stiff} if for every $n \ge d$ and every subgraph $H'$ of $Q_n$ admitting an isomorphism $\phi$ to $H$, there exists a bijection $f: D(H) \rightarrow D(H')$ such that for each $i \in D(H)$, $\phi(E_i(H')) = E_{f(i)}(H)$. In other words, the partition $\phi(E_1(H')),\ldots, \phi(E_n(H'))$ of $E(H)$ is the same partition as $E_1(H), \ldots, E_n(H)$. 

\begin{proof}[Proof of Theorem~\ref{th:counter}]
For $k \ge 5$, we will show there exists a graph $H_k$ with the following properties.
\begin{enumerate}
\item $Q_{k-1} \subseteq H_k \subseteq Q_{k}$.
\item $|E_i(H_k)| = 2^{k-1} - 1$ for all $i \in \{1,\ldots,k\}$.
\item Every vertex of $H_k$ has degree at least $k -1$.
\item $H_k$ is stiff. 
\end{enumerate}

Given such a graph $H_k$, we show that for no $n \ge k$ can the edges of $Q_n$ be partitioned into edge-disjoint copies of $H_k$. Fix $n$ and consider $E_1(Q_n)$. By definition of $Q_n$, $E_1(Q_n)$ contains $2^{n-1}$ edges. As $H_k$ is stiff, any copy of $H_k$ in $Q_n$ intersects $E_1(Q_n)$ on precisely $2^{k-1} - 1$ edges. However, $|E_1(Q_n)| = 2^{n-1}$, which is not divisible by $ 2^{k-1} - 1$. Therefore it is not possible to cover every edge of $E_1(Q_n)$ using edge-disjoint copies of $H_k$ and hence it is not possible to cover $E(Q_n)$. 

Thus it suffices to show the existence of $H_k$, for $k \ge 5$. We will construct the graphs $H_k$ inductively. Define $H_5$ to be the graph obtained from $Q_5$ by deleting the edges between: $(0,1,1,0,1)$ and $(1,1,1,0,1)$, $(1,0,1,1,0)$ and $(1,1,1,1,0)$, $(1,0,0,0,1)$ and $(1,0,1,0,1)$, $(1,1,0,0,0)$ and $(1,1,0,1,0)$, and finally $(1,0,0,1,0)$ and $(1,0,0,1,1)$. The graph $H_5$ is depicted in Figure~\ref{H4}. Observe that $H_5$ is a graph obtained from $Q_5$ by removing exactly one edge in each direction in such a way that no vertex has degree 3. 

\begin{figure}[htbp]
\centering
\includegraphics[width=0.6\textwidth]{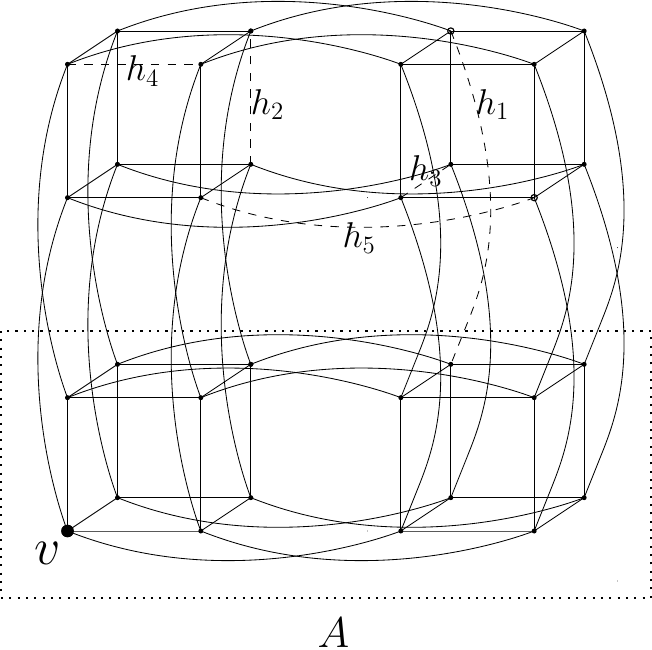}
\caption{The graph $H_5$. The vertex $v = (0,0,0,0,0)$ is emboldened. The edges $h_1,\ldots,h_5$, represented by dashed lines, are the edges removed from $Q_5$. The copy of $Q_4$ is induced by the vertices in $A$.}
\label{H4}
\end{figure}
Suppose we have constructed $H_k$. Now construct $H_{k+1}$ as follows. Let $H_{k+1}$ be a subgraph of $Q_{k+1}$ where: 
\begin{itemize}
\item[(a)] the vertices $V_0 := \{(v_1,\ldots,v_{k+1}) \in V(Q_{k+1}): v_1 = 0\}$ induce a copy of $H_k$; 
\item[(b)] the vertices $V_1:= \{(v_1,\ldots,v_{k+1}) \in V(Q_{k+1}): v_1 = 1\}$ induce a copy of $Q_{k}$, and;
\item[(c)]  every edge of $Q_{k+1}$ between $V_0$ and $V_1$ is contained in $H_{k+1}$ except one which is not incident to any non-edge of $H_{k+1}[V_0]$.
\end{itemize}
 It is always possible to satisfy (c) as there are precisely $k$ non-edges in $H_{k+1}[V_0]$ and these edges are incident with $2k < 2^k$ vertices. Observe that there are many valid choices of $H_{k+1}$: it is not unique.
 
 By construction each $H_k$ satisfies properties 1,2 and 3. The only thing remaining to prove is that each $H_k$ is stiff. 
 
 $H_k$ contains a subgraph $A$ isomorphic to $Q_{k-1}$. Let $v$ be a vertex of $A$ of degree $k$ in $H_k$ (which exists by construction). The edges $\{e_1,\ldots,e_k\}$ incident to $v$ have different directions. Without loss of generality, suppose that $D(e_i) = i$ for each $i \in [k]$. Now consider a subgraph $H'$ of $Q_n$ such that $\phi$ is an isomorphism from $H$ to $H'$. Up to permuting co-ordinates, we may assume that $D(\phi(e_i))= i$ for each $i \in [k]$. Say that an edge of $H$ is \emph{good} if $D(e) = D(\phi(e))$. To prove that $H$ is stiff it suffices to show that every edge in $H$ is good. 
 
 By definition, $e_1,\ldots,e_k$ are good. Suppose $v_1v_2v_3v_4$ is a copy of $C_4$ in $H$. If $v_1v_2$ and $v_2v_3$ are good, then so are both $v_3v_4$ and $v_4v_1$. Thus it suffices to exhibit an ordering $f_1,f_2,\ldots$ of the edges of $H_k$ such that  $f_i:= e_i$ for $i \in [k]$ and for each $i \ge k + 1$, there exist $j_1,j_2 < i$ such that $f_{j_1}$ and $f_{j_2}$ are edges of different directions and such that $f_i$ is contained in a copy of $C_4$ with $f_{j_1}$ and $f_{j_2}$. 
 
As $A$ is isomorphic to $Q_{k-1}$, clearly there exists such an ordering of the edges in $A$. Given this, it is not difficult to see that this ordering can be extended to an ordering of $E(H_k)$ with the required properties (as, by construction, every vertex in $H_k$ has degree at least $k-1$). Thus, every edge in $H_{k}$ is good and $H_{k}$ is stiff. This completes the proof of the Theorem. 
 
 \end{proof}

\section{Conclusion}\label{thatsall}

As we have shown in Section~\ref{tor2}, the extension of Conjecture~\ref{vconj} from the hypercube to arbitrary Cartesian products is false (see Theorem~\ref{th:unpackabletori}). However, it is natural to hope that there might be some extension of Theorem~\ref{imre} to general infinite vertex-transitive graphs. Unfortunately, this also turns out to be false. 

\begin{thm}
There exists an infinite vertex-transitive graph $G$ and a finite subgraph $H \subseteq G$ such that $G^n$ does not admit a perfect induced $H$-packing for any $n$.
\end{thm}
We provide a sketch of the proof, as a lot of the ideas involved are analogous to ideas developed in detail in the proof of Theorem~\ref{th:unpackabletori}.

\begin{proof}[Sketch proof]
Let $a < b$ be odd coprime integers and let $G_1$ be $C_{ab}$. Define $G$ to be the Cartesian product of $G_1$ with $\mathbb{Z}$ (that is, $G:= G_1 \Box \mathbb{Z}$). So $G^n := G_1^n \Box \mathbb{Z}^n$.

Say that a \emph{line} of $G^n$ is a subgraph of $G^n$ obtained by fixing $2n-1$ co-ordinates (and letting the unfixed co-ordinate vary). Let $H_1$ be a copy of $G_1$. For any $n$, the only way to embed $H_1$ in $G^n$ is as a line. Indeed, if we first embed $v_1$ at some vertex $x$ and then go around the cycle embedding each edge sequentially, then after embedding $ab$ edges we arrive back at $x$. Each edge embedded corresponds to a change in one coordinate of $+1$ or $-1$. So either a coordinate increases and decreases exactly the same number of times (so in total changes an even number of times) as we move around the cycle, or it changes $ab$ times in the same direction (in which case it arrives back where it started). As $ab$ is odd, the only way for $H_1$ to be embedded is as a line, where some coordinate corresponding to a copy of $G_1$ in $G^n$ is varied.

Now let $H$ be the subgraph of $G_1^2$ induced by the vertices in (\ref{Hdef}). In the proof of Theorem~\ref{th:unpackabletori}, we saw that $H$ contains \emph{vertical} and \emph{horizontal} cycles which must embed as lines. As in the proof of Theorem~\ref{th:unpackabletori}, once these cycles are embedded, the embedding of the rest of $H$ is determined. In particular, the only coordinates in the embedding of $H$ which vary are those which vary on the embedding of the cycles. In particular, the coordinates corresponding to $\mathbb{Z}^n$ remain fixed in the embedding of $H$.

Therefore, if there were to exist a perfect $H$-packing of $G^n$ for some $n$, (forgetting appropriate edges) we get a packing of $C_k^n$ with the $H$ from Theorem~\ref{th:unpackabletori}. However, such a packing cannot exist (as shown by Theorem~\ref{th:unpackabletori}).
\end{proof}

In this paper we have determined that for any even $k$ the base conditions on $H$ are sufficient to guarantee a perfect induced $H$-packing of $C_k^n$ for $n$ sufficiently large (Theorem~\ref{th:packabletori}). We have also shown that when $k$ is odd and not a prime power, the base conditions are not sufficient (Theorem~\ref{th:unpackabletori}). Thus it is natural to wonder what happens when $k$ is an odd prime power.

\begin{conj}
Let $k$ be an odd prime power and $m\ge1$.  Suppose that  $H$ is an induced subgraph of $C_k^m$ such that $|V(H)|$ divides $k^m$. Then there exists $n_0$ such that, for all $n \ge n_0$, $C_k^n$ admits a perfect induced $H$-packing.
\end{conj}

In Section~\ref{sec:edges} we considered the question of whether, for a given graph $H$, there exists $d$ such that it is possible to cover the edges of $Q_d$ with edge-disjoint copies of $H$. We showed that the necessary conditions that $H$ is a subgraph of $Q_d$ and that $|E(H)|$ divides $|E(Q_d)|$ are not sufficient. When $H$ is a path and 
suitable divisibility conditions hold, 
Erde~\cite{erde} and Anick and Ramras~\cite{anram} independently showed that 
there is a partition of $E(Q_n)$ into edge-disjoint copies of $H$. It would be interesting to prove analogous results for other graphs. 

\begin{qn}
For which graphs $G$ does there exist some $d$ such that the edges of $Q_d$ can be partitioned into edge disjoint copies of $d$?\footnote{We remark that this question differs slightly from the published version of this paper.}
\end{qn}

Naturally, in order for this to be possible we require that $|E(G)|$ divides $|E(Q_d)|$ and that it is possible to express $d-1$ as the sum of degrees of vertices in $G$. For example, it seems likely that this should be sufficient when $G$ is an even cycle.

\begin{ack}
The authors would like to thank the anonymous referees for their very helpful comments.
\end{ack}

\bibliography{tori}
  \bibliographystyle{amsplain}
\end{document}